\documentclass{amsart}

\usepackage{amsmath,amsthm,mathrsfs,amssymb,graphicx,centernot}
\usepackage[all]{xy}
\usepackage{tikz}
\usetikzlibrary{shapes,snakes}
\usepackage[pdfborder={0 0 0},backref=false,colorlinks=true,linktocpage=true]{hyperref}
\usepackage[natbib=true,doi=false,isbn=false,url=false,style=numeric,maxbibnames=99,backend=bibtex]{biblatex}
\PassOptionsToPackage{hyphens}{url}

\renewbibmacro{in:}{\ifentrytype{article}{}{\printtext{\bibstring{in}\space}}}
\renewbibmacro*{volume+number+eid}{\printfield{volume}\printfield{number}\setunit{\addcomma\space}\printfield{eid}}
\DeclareFieldFormat[article,inbook,incollection,inproceedings,patent,thesis,unpublished]{title}{#1}
\DeclareFieldFormat[article]{volume}{\mkbibbold{#1}}
\DeclareFieldFormat[article]{number}{\mkbibparens{#1}}
\DeclareFieldFormat[article]{pages}{#1}

\newtheorem{theorem}{Theorem}[section]
\newtheorem{proposition}[theorem]{Proposition}
\newtheorem{lemma}[theorem]{Lemma}
\newtheorem{observation}[theorem]{Observation}

\theoremstyle{definition}
\newtheorem{definition}[theorem]{Definition}
\newtheorem*{framework}{Framework}

\newtheorem{question}[theorem]{Question}

\newcommand{\res}{\mathbin{\upharpoonright}}
\newcommand{\ran}{\operatorname{ran}}
\newcommand{\dom}{\operatorname{dom}}

\newcommand{\seq}[1]{\langle #1 \rangle}

\newcommand{\nimplies}{\centernot\implies}

\newcommand{\from}{\leftarrow}

\newcommand{\RCA}{\mathsf{RCA}}

\newcommand{\cor}{{\mathrm{cor}}}
\newcommand{\mr}{{\mathrm{mr}}}
\newcommand{\mf}{{\mathrm{mf}}}
\newcommand{\ubfb}{{\mathrm{ubfb}}}
\newcommand{\cf}{{\mathrm{cf}}}
\newcommand{\ii}{{\mathrm{ii}}}
\newcommand{\gen}{{\mathrm{g}}}
 
\begin{document}

\title{Notions of robust information coding}

\author{Damir D. Dzhafarov}
\address{Department of Mathematics\\
University of Connecticut\\
196 Auditorium Road\\
Storrs, Connecticut 06269 U.S.A.}
\email{damir@math.uconn.edu}

\author{Gregory Igusa}
\address{Department of Mathematics\\
University of Notre Dame\\
255 Hurley Hall\\
Notre Dame, Indiana 46556 U.S.A.}
\email{gigusa@nd.edu}

\thanks{The authors are grateful to Christopher Porter and Theodore Slaman for helpful insights and discussions. Dzhafarov was supported in part by an NSF Postdoctoral Fellowship, and Igusa was supported in part by RTG grant EMSW21-RTG-0838506.}

\maketitle

\begin{abstract}
We introduce and study several notions of computability-theoretic reducibility between subsets of~$\omega$ that are ``robust'' in the sense that if only partial information is available about the oracle, then partial information can be recovered about the output. These are motivated by reductions between~$\Pi^1_2$ principles in the context of reverse mathematics, and also encompasses generic and coarse reducibilities, previously studied by Jockusch and Schupp~\cite{JS-2012}.
\end{abstract}

\section{Introduction}\label{S:intro}

The study of \emph{relative} computation, of which subsets of~$\omega$ are Turing reducible to, or computable from, which others, has been the principal interest of classical computability theory almost from the subject's inception. The concept of oracle Turing machines is natural and powerful. But working with it quickly leads one to discover that in many cases a set is actually computable from another as a result of a much stronger type of reduction. And indeed, this observation has led to the identification and independent study of a myriad of refinements and modifications of ordinary Turing reducibility, each having to do with a particular method of how information can be coded and decoded.

In this article, we introduce several definitions of computability-theoretic reducibility based on robustness of information coding. The background question for us here is the following: how can information be coded so that it can be recovered with only partial access to the oracle, or with unusual rules for accessing the oracle? This is motivated in large part by the study of~$\Pi^1_2$ principles in reverse mathematics, i.e., statements of the form, ``For every set~$A$ (called an \emph{instance} of the principle) with certain arithmetical properties, there is a set~$B$ (called a \emph{solution} to the instance~$A$) with certain arithmetical properties with respect to~$A$''. Proving an implication from some such theorem (over the base theory~$\RCA_0$) typically involves building an instance~$A$ of the principle such that every solution~$B$ to~$A$ is computationally powerful. This can be viewed as a way of coding information into the set~$A$, although the method of decoding (i.e., first passing to a solution) may make it appear rather indirect. In particular, as the set of solutions may be closed under various set-theoretic and computability-theoretic operations (e.g., as is not uncommon, under infinite subset), it is important to know how to build~$A$ in such a way that the information coded into it can be recovered from \emph{any} solution, not just a particular one.

Our aim, then, is to understand the coding strategies that permit this type of ``persistent'' information retrieval. This leads us to formulate the following reducibilities between sets of numbers:
\begin{itemize}
\item mod-finite reducibility,~$\leq_\mf$;
\item cofinite reducibility,~$\leq_\cf$;
\item generic reducibility,~$\leq_\gen$;
\item coarse reducibility,~$\leq_\cor$;
\item mod-recursive reducibility,~$\leq_\mr$;
\item infinite-information reducibility,~$\leq_\ii$;
\item use-bounded-from-below reducibility,~$\leq_\ubfb$;
\end{itemize}
All of these, formally defined in Section \ref{S:def}, have the general form described by the following framework.

\begin{framework}
Given a notion of largeness, a set~$B$ is reducible to a set~$A$ if every set that agrees with~$A$ on a large domain uniformly computes a set that agrees with~$B$ on a large domain. 
\end{framework}

\noindent The idea is that if~$A$ computes~$B$ in a sufficiently robust manner, then changing (or losing) a small number of the bits of the oracle should still allow us to recover most of the output. In a sense, then, we are looking at certain kinds of error-correcting or self-redundant codes, in which no single bit is essential for the computation of too many of the rest.

As a simple-minded example, suppose~$B$ is~$1$-reducible to~$A$, written~$B \leq_1 A$, and meaning~$B = \{n \in \omega: f(n) \in A\}$ for some computable one-to-one function~$f$. Then every set cofinitely equal to~$A$ computes a set cofinitely equal to~$B$, via the same~$1$-reduction. Thus, if ``large'' is taken to mean ``cofinite'', we have that~$B$ is reducible to~$A$ in the sense of the above framework.

In this article, we confine our interest to the seven reducibilities just described, though others can certainly be fitted into our framework. (A discussion of some possible variations and extensions appears at the end of Section~\ref{S:def}.) Likewise, our results here are not meant to be definitive, but rather to develop, also with a view towards future work, a set of techniques with which our reducibilities, and potentially others, can be studied. This article is organized as follows. In Section \ref{S:def} we formally define our reducibilities, and more precisely describe our framework for studying them. In Section \ref{S:imps} we prove Theorem \ref{T:imps}, establishing how our reducibilities compare to one another, and exhibit embeddings of some of these reducibilities into others. Section \ref{S:nonimps} concerns non-implications, and shows that no relationships between our reducibilities hold other than those of Theorem \ref{T:imps}. Finally, in Section \ref{S:inf} we focus on~$\leq_\ii$, and prove that this reducibility admits \emph{maximal} pairs.

\section{Definitions}\label{S:def}

Unless otherwise noted, we shall follow standard notation and terminology employed in the literature. This includes, for a partial function~$f$ on~$\omega$, writing~$f(n) \downarrow$ if~$n$ is in the domain of~$f$. If~$f$ and~$g$ are two partial functions, we shall write~$f(n) \simeq g(n)$ if~$f(n) = g(n)$ provided~$n$ is in the domain of both functions, and we shall write~$f \simeq g$ if~$f(n) \simeq g(n)$ for all~$n$. We shall also always identify sets with their characteristic functions. Though we shall be dealing with various reducibilities, and hence implicitly with various join and meet operators on sets, we shall reserve~$\oplus$ for the ordinary Turing join of sets:~$A\oplus B=\{ 2n:n\in A\}\cup\{ 2n+1:n\in B\}$. We refer the reader to Soare \cite{Soare-1987} for complete background on computability theory, and to Downey and Hirschfeldt \cite{DH-2010} for background on algorithmic randomness, which several of our results will relate to. In addition, we refer the reader to Sacks \cite{Sacks-1990} for background on higher recursion theory, which will be applicable to our work in Section \ref{S:inf}.

We now pass to the definitions of the reducibilities we shall be concerned with.

Our first examples are motivated by recent work studying generic computability and coarse computability,~\cite{JS-2012},~\cite{DJS-2013},~\cite{Igusa-2013}. These notions concern being able to correctly compute most of a set in the following sense.

\begin{definition}\label{D:density1}
Let~$A$ be a subset of the natural numbers. Then~$A$ has \emph{density~1} if the limit of the densities of its initial segments is 1, or in other words, if~$\lim_{n\rightarrow\infty}\frac{|{A\cap n}|}{n}=1$. In this case, we will frequently say that~$A$ is \emph{density-1.}
\end{definition}

\begin{definition}[\cite{JS-2012}, Definitions~1.4 and~2.3]
\
\begin{enumerate}
\item A set~$B$ is \emph{generically computable} if there is a partial computable function~$\Phi$ whose domain has density~$1$ such that~$\Phi \simeq B$.
\item A set~$B$ is \emph{coarsely computable} if there is a total computable function~$\Phi$ such that~$\{n \in \omega: \Phi(n) = B(n)\}$ has density~$1$.
\end{enumerate}
\end{definition}
\noindent In other words, a generically computable set is given by a computation that never lies, but might not halt on a domain of density~$0$. A coarsely computable set, by contrast, is given by a computation that always halts, but might lie on a domain of density~$0$.

Perhaps the most obvious way to relativize generic computability is to say that~$B$ is \emph{generically computable from}~$A$ if there is a functional~$\Phi$ such that~$\Phi^A \simeq B$ with domain of density~$1$. This is not by itself a reducibility between sets because it is highly non-transitive. In fact, Igusa~\cite{Igusa-2013} has shown that for any reflexive binary relation~$R$ on~$\omega$ there exists a class~$\mathcal{C}$ of subsets of~$\omega$ such that~$(\omega,R)$ is isomorphic to the structure of~$\mathcal{C}$ under the relation of generic computability.

Thus, we shall work instead with the following relativizations of coarse and generic computability. The latter appears also in \cite[Definition 4.3]{JS-2012}.

\begin{definition}\label{D:coar_formal}
A set~$B$ is \emph{coarsely} reducible to a set~$A$, written~$B \leq_\cor A$, if there exists a Turing functional~$\Phi$ such that for any set~$C$, if~$\{n \in \omega : C(n) = A(n)\}$ has density~$1$ then~$\Phi^C$ is total and~$\{n \in \omega: \Phi^C(n) = B(n)\}$ has density~$1$.
\end{definition}

Because generic computability intrinsically involves partial computations, if we wish to define generic reducibility, we must first define what a partial oracle is.

\begin{definition}
Let~$A$ be a set. A \emph{partial oracle,}~$(A)$, for~$A$ is a set of ordered triples~$\langle n,x,l\rangle$ such that:
\begin{itemize}
\item $\exists l\big(\langle n,0,l\rangle\in (A)\big)\Longrightarrow n\notin A$;
\item $\exists l\big(\langle n,1,l\rangle\in (A)\big)\Longrightarrow n\in A$.
\end{itemize}
\end{definition}

\noindent We regard~$(A)$ as a partial function that maps~$n$ to~$x$, and that takes~$l$ steps to converge. Under this interpretation, a partial oracle is a partial function~$(A) : \omega \to \{0,1\}$ such that~$(A) \simeq A$. When we refer to querying the~$n$th bit of~$(A)$ we are actually initiating a subalgorithm that searches for some~$x,l$ such that~$\langle n,x,l\rangle\in (A)$. While waiting, a computation is allowed to perform other calculations and even to query other bits of the oracle. The domain of~$(A)$ is the set of~$n$ so that there exist such~$x$ and~$l$.

The primary purpose of the~$l$ in the formal definition is to ensure that, when working with a partial oracle, one does not know whether or not the oracle will give an output on some given input. However, in our notation in this paper, we will follow the intuitive concept of a partial oracle as a partial function, rather than the formal definition as a set of triples. Under this convention, a set~$A$ is a partial oracle for itself that immediately halts on every input.

We may now define generic reducibility. This is not the original definition by Jockusch and Schupp, which used enumeration operators. However, the two definitions are equivalent, as shown in \cite[Proposition 3.10]{Igusa-2013}.

\begin{definition}\label{D:gen_formal}
A set~$B$ is \emph{generically} reducible to a set~$A$, written~$B \leq_\gen A$, if there exists a Turing functional~$\Phi$ such that~$\Phi^A = B$, and for any partial oracle~$(A)$ for~$A$ with density-$1$ domain,~$\Phi^{(A)}$ has density-$1$ domain, and~$\Phi^{(A)} \simeq B$.
\end{definition}

\noindent Note that generic and coarse reducibility agree with generic and coarse computability when~$A = \emptyset$.

Our next reducibilities are analogues of generic and coarse reducibility that allow only finite omission/error. They are motivated in part as a generalization of 1-reducibility, as was mentioned in the previous section, and in part to help understand the difference between partial oracles and partially incorrect oracles by examining the difference in a simplified framework that does not need to reference asymptotic density.

\begin{definition}\label{D:mf_formal}
A set~$B$ is \emph{mod-finite} reducible to a set~$A$, written~$B \leq_\mf A$, if there exists a Turing functional~$\Phi$ such that~$\Phi^A=B$, and for any set~$C$, if~$\{n \in \omega : C(n) = A(n)\}$ is cofinite then~$\Phi^C$ is total and~$\{n \in \omega: \Phi^C(n) = B(n)\}$ is cofinite.
\end{definition}

\begin{definition}\label{D:cf_formal}
A set~$B$ is \emph{cofinitely} reducible to a set~$A$, written~$B \leq_\cf A$ if there exists a Turing functional~$\Phi$ such that~$\Phi^A = B$, and for any partial oracle~$(A)$ for~$A$ with cofinite domain,~$\Phi^{(A)} \simeq B$ with cofinite domain.
\end{definition}

\noindent Note that for these two reducibilities, and also for the upcoming reducibility, the uniformity hypothesis (that a single~$\Phi$ must work for all possible mod-finite, or cofinite oracles for~$A$) is essential, because without it, both of the reducibilities would be equivalent to Turing reducibility.

Our next reducibilities do not involve true notions of largeness, in the commonly accepted sense of being closed under intersection and superset. The first takes ``large'' to be ``computable''.

\begin{definition}\label{D:mr_formal}
A set~$B$ is \emph{mod-recursive} reducible to a set~$A$, written~$B \leq_\mr A$, if there exists a Turing functional~$\Phi$ such that~$\Phi^A=B$, and for any set~$C$, if~$\{n \in \omega : C(n) = A(n)\}$ is computable then~$\Phi^C$ is total and~$\{n \in \omega: \Phi^C(n) = B(n)\}$ is computable.
\end{definition}

\noindent This sort of reducibility is somewhat uncommon in computability theory, not counting fairly trivial cases. However, it is quite common in real-world information coding, where, frequently, changes in the input yield predictable changes in the output. It also provides an interesting example that shows that our framework for computation does not require our notion of largeness to be closed under superset. (A mod-recursive computation must be correct exactly on a computable set, not just correct on a set containing a computable set.)

The second notion takes ``large'' to be ``infinite''. Being infinite is arguably the most salient feature of any notion of largeness, and as such the corresponding reducibility is among the most natural for our purposes. By contrast, the results we obtain about it in Section \ref{S:inf} also make it the most counterintuitive to understand.

\begin{definition}\label{D:ii_formal}
A set~$B$ is \emph{infinite-information} reducible to a set~$A$, written~$B \leq_\ii A$, if there exists a Turing functional~$\Phi$ such that for any partial oracle~$(A)$ for~$A$ with infinite domain,~$\Phi^{(A)} \simeq B$ with infinite domain.
\end{definition}

\noindent This reduction is motivated by, and closely tied to, coding and de-coding techniques commonly used in reverse mathematics, particularly in the study of Ramsey's theorem and related combinatorial principles. For instance, knowing an infinite homogeneous set for a computable stable~$2$-coloring corresponds precisely to knowing an infinite set of bits of a~$\Delta^0_2$ set.

Finally, we consider a reducibility defined in terms of restricting our computation techniques, rather than in terms of an explicit notion of largeness with which to measure our oracles and outputs. This reducibility does not explicitly fit into our framework, but it is closely tied to both mod-finite and cofinite reductions.

\begin{definition}
A set~$B$ is \emph{use-bounded-from-below} reducible to a set~$A$, written~$B \leq_\ubfb A$, if there is a Turing functional~$\Phi$ such that~$\Phi^A = B$ and for every~$m$ and all sufficiently large~$n$, the functional does not query~$A(m)$ in computing~$\Phi^A(n)$.
\end{definition}

\noindent Normally, the use of a computation is thought of in terms of the largest element of the oracle queried to obtain an output, but here, we care instead about the \emph{smallest} such query. Concordantly, we shall avoid, when necessary, various usual conventions of oracle computation that may interfere with the analysis of~$\leq_\ubfb$, such as identifying the use with an initial segment of the oracle.

Each of our reducibilities obviously gives rise to a degree structure. We shall call these the \emph{mod-finite}, \emph{coarse}, \emph{cofinite}, \emph{generic}, \emph{mod-recursive}, \emph{infinite-information}, and \emph{use-bounded-from-below degrees}, respectively.

There are several natural ways in which our definitions above can be modified and extended. The first comes from considering non-uniform versions, in which the functional~$\Phi$ in the definition of the reducibility is allowed to depend on the specific oracle (or partial oracle). Of course, this does not make sense for use-bounded-from-below reducibility, while for mod-finite and cofinite reducibilities, it is easy to check that the non-uniform versions reduce to ordinary Turing reducibility. We do not explicitly study non-uniform versions of generic or coarse reducibilities here, but many of the results hold in either version. The situation for non-uniform infinite-information reducibility is open.

The next natural modification of our definitions is in whether or not we insist that~$\Phi^A = B$. We do not, in this article, consider versions of generic, coarse, and infinite-information reducibility where this hypothesis is added. Indeed, doing so for generic and coarse reducibility would mean the two no longer extend generic and coarse computability, and doing so for infinite-information reducibility would lose the connection to reverse mathematics mentioned after Definition \ref{D:ii_formal}. However, these versions might be interesting to study in their own right. For mod-finite, cofinite, and mod-recursive reducibilities, the addition of this hypothesis is actually superfluous, and we add it only because doing so makes the reducibilities more convenient to work with.

\begin{observation}
\
\begin{enumerate}
\item Let~$A,B,\Phi$ be such that~$\Phi$ witnesses a cofinite reduction of~$B$ to~$A$ without the hypothesis that~$\Phi^A = B$ (i.e., if~$(A)$ is any partial oracle for~$A$ with cofinite domain then~$\Phi^A \simeq B$ with cofinite domain). Then~$B$ is cofinitely reducible to~$A$ (in the full sense of Definition \ref{D:cf_formal}).

\item The same holds for mod-recursive reducibility.

\item The same holds for mod-finite reducibility.

\end{enumerate}

\end{observation}

\begin{proof}

Fix~$A,B,\Phi$ be as in~1. Then~$\widetilde{\Phi}$ witnesses a cofinite reduction of~$B$ to~$A$, where~$\widetilde{\Phi}$ is defined by
\[
\widetilde{\Phi}^C(n) =
\begin{cases}
\Phi^C(n) &\text{if } \Phi^A(n) \downarrow,\\
B(n) &\text{if } \Phi^A(n) \uparrow.
\end{cases}
\]
Note that~$\{n:\Phi^A(n)\uparrow\}$ is finite, so the values of~$B$ on that set can be coded directly into~$\widetilde\Phi$.

For case~2, assume~$\Phi$ instead witnesses a mod-recursive reduction of~$B$ to~$A$ without the hypothesis that~$\Phi^A = B$. Then~$\widetilde{\Phi}$ witnesses a mod-recursive reduction of~$B$ to~$A$, where~$\widetilde{\Phi}$ is defined by
\[
\widetilde{\Phi}^C(n) =
\begin{cases}
\Phi^C(n) & \Phi^A(n) \downarrow = B(n),\\
1 - \Phi^C(n) & \Phi^A(n) \downarrow \neq B(n),
\end{cases}
\]
Now~$\{n:\Phi^A(n)\neq B(n)\}$ is computable, so~$\widetilde\Phi$ is a computable modification of~$\Phi$.

Case~3 can be handled in either of these ways.
\end{proof}

Finally, we can look further at the differences between the set versus partial oracle versions of our reducibilities, as in mod-finite versus cofinite, and coarse versus generic. The set version of infinite-information reducibility is trivial. Any set is either infinite or coinfinite, and so either the computation that outputs~$0$ everywhere or the computation that outputs~$1$ everywhere is correct infinitely often. Thus, every set would be computable from the empty oracle. The partial oracle version of mod-recursive is also trivial, because the empty set is computable, and it is trivial to produce a partial computation of any real whose domain is empty. If one were to demand that the domains of the input and output be computable and infinite, then this reducibility would share many qualities of infinite-information reducibility. However, we do not explicitly study this version.

\section{Embeddings and implications}\label{S:imps}

We begin with a following theorem, establishing how our reducibilities compare to one another, as well as to~$1$-reducibility and Turing reducibility. For notational convenience, if~$\leq_*$ and~$\leq_{*'}$ are any two reducibilities between sets, we shall say~$\leq_*$ \emph{implies}~$\leq_{*'}$, and write~$\leq_* \implies \leq_{*'}$, if for all sets~$A$ and~$B$,
\[
A \leq_* B \implies A \leq_{*'} B.
\]

\begin{theorem}\label{T:imps}
\
\begin{enumerate}
\item We have
\[
\leq_1 \implies \leq_\mr \implies \leq_T,
\]
and
\[
\leq_1 \implies \leq_\mf \implies \leq_\ubfb \implies \leq_\cf \implies \leq_T.
\]
\item No additional implications hold between any of the above reducibilities, or between them and~$\leq_\gen$,~$\leq_\cor$, and~$\leq_\ii$.
\end{enumerate}
\end{theorem}

\begin{proof}[Proof of part~1]
\

\medskip
\noindent ($\leq_1 \implies \leq_\mr$) Suppose~$B \leq_1 A$ via the computable function~$f$, and let~$\Phi$ be the corresponding Turing functional. Then for any set~$C$, we have
\[
\{n : \Phi^C(n) = B(n)\} = \{n : C(f(n)) = A(f(n)) \} \leq_T \{n : C(n) = A(n)\}.
\]
Hence,~$B \leq_\mr A$ via~$\Phi$.

\medskip
\noindent ($\leq_\mr \implies \leq_T$) Clear.

\medskip
\noindent ($\leq_1 \implies \leq_\mf$) Suppose~$B \leq_1 A$ via the computable function~$f$, and let~$\Phi$ be the corresponding Turing functional. Then if~$C =^* A$ we have
\[
\Phi^C = \{n : f(n) \in C\} =^* \{n : f(n) \in A\} = B.
\]

\medskip
\noindent ($\leq_\mf \implies \leq_\ubfb$) Suppose~$B \leq_\mf A$ via the functional~$\Phi$. Define a new functional~$\Psi$ that works as follows. Given a set~$S$ and number~$n$, let~$S_0,\ldots,S_{2^n-1}$ be the~$2^n$ many possible sets obtained by changing~$S$ on the first~$n$ many bits. Now~$\Psi^S(n)$ proceeds by first computing~$\Phi^{S_i}(n)$ for each~$i < 2^n$ without querying any bits of~$S$ below~$n$, and then querying the bits of~$S$ below~$n$ one by one in decreasing order until all but one of the values from~$\Phi^{S_0}(n),\ldots,\Phi^{S_{2^n-1}}(n)$ are eliminated. It is not difficult to see that~$\Psi^A = B$. Furthermore, since~$\Phi^A$ is a mod-finite reduction, for each~$m$ there is an~$m'$ such that each of~$\Phi^{A_0},\ldots,\Phi^{A_{2^m-1}}$ agree from~$m'$ on. Thus, if~$n \geq m'$ then in particular~$\Psi^{A}(n)$ does not query any bit of the oracle below~$m$. This means that~$\Psi^A$ is a uniform-bounded-from-below reduction. 

\medskip
\noindent ($\leq_\ubfb \implies \leq_\cf$) Assume~$A \leq_\ubfb B$ via the functional~$\Phi$. Then~$\Phi$ also witnesses that~$A \leq_\cf B$, provided we declare~$\Phi^{(A)}(n)$ to be undefined if during its computation~$(A)$ is undefined on some queried bit. Now suppose~$(A)$ is a partial oracle for~$A$ with cofinite domain. Then if~$n$ is sufficiently large, the fact that~$\Phi$ is a use-bounded-from-below reduction means it will only query bits in the domain of~$(A)$ in computing~$\Phi^{(A)}(n)$. Hence, on all sufficiently large inputs,~$\Phi^{(A)}$ will converge and equal~$\Phi^A = B$.

\medskip
\noindent ($\leq_\cf \implies \leq_T$) Immediate from the definition.
\end{proof}

\noindent We delay the proof of part~2 of the theorem to the next section.

It follows immediately that if~$\leq_*$ is any of~$\leq_\mr$,~$\leq_\mf$,~$\leq_\cf$, or~$\leq_\ubfb$ then for every set~$A$ there is a~$B$ such that~$B \nleq_* A$. (Namely, take any~$B \nleq_T A$.) Thus, these reducibilities are non-trivial. The following result strengthens this observation, and extends it to~$\leq_\cor$,~$\leq_\gen$, and~$\leq_\ii$.

\begin{proposition}\label{P:randomgenerics}
Let~$\leq_*$ be any of our reducibilities, and let~$A$ be any set. If~$B$ is weakly~$1$-generic or~$1$-random relative to~$A$, then~$B \nleq_* A$.
\end{proposition}

\begin{proof}
As above, this follows for~$\leq_\mr$,~$\leq_\mf$,~$\leq_\cf$, and~$\leq_\ubfb$ by Theorem \ref{T:imps}. To see it for~$\leq_\cor$, note that for every functional~$\Phi$ such that~$\Phi^A$ is total the set of~$\sigma \in 2^{<\omega}$ with
\[
\frac{|\{n < |\sigma|: \Phi^A(n) = \sigma(n) \}|}{|\sigma|} < \frac{1}{2}
\]
is computably enumerable relative to~$A$ and dense, and as such is infinitely often met by any set weakly~$1$-generic relative to~$A$. Hence, no such set is coarsely reducible to~$A$. For~$\leq_\ii$, observe that if~$(A)$ is a given partial oracle for~$A$ with infinite domain such that~$\Phi^{(A)}$ also has infinite domain, then the collection of strings~$\sigma$ such that~$\Phi^{(A)}(n) \downarrow \neq \sigma(n)$ for some~$n < |\sigma|$ is dense and computably enumerable relative to~$A$. The proof for~$\leq_\gen$ is the same.
\end{proof}

We shall use the following definitions here and in the rest of this paper. Given a set~$S$, define
\[
\mathcal{R}(S) = \{ 2^n m : n \in S,~m \textrm{ odd} \}
\]
and
\[
\widetilde{\mathcal{R}}(S) = \{ 2^n + m : n \in S,~m < 2^n \}.
\]
In other words, a number belongs to~$\mathcal{R}(S)$ if and only if the highest power of~$2$ that divides it belongs to~$S$, while a number belongs to~$\widetilde{\mathcal{R}}(S)$ if and only if the highest power of~$2$ less than or equal to it belongs to~$S$. Jockusch and Schupp \cite[Lemma 4.6]{JS-2012} showed that the map~$S \mapsto \mathcal{R}(S)$ induces an embedding of the Turing degrees into the generic degrees. We use the map~$S~\mapsto~\widetilde{\mathcal{R}}(S)$ to exhibit two similar embeddings. The key distinction is that that any partial oracle for~$\mathcal{R}(S)$ with domain of density~$1$ uniformly computes~$S$ (see \cite[Observation 2.11]{JS-2012}), while any such oracle for~$\widetilde{\mathcal{R}}(S)$ only uniformly computes a partial oracle for~$S$ with cofinite domain.

\begin{proposition}\label{P:embeddings}
\
\begin{enumerate}
\item The map~$S \mapsto \widetilde{\mathcal{R}}(S)$ induces an embedding of the cofinite degrees into the generic degrees.
\item The map~$S \mapsto \widetilde{\mathcal{R}}(S)$ induces an embedding of the mod-finite degrees into the coarse degrees.
\item The map~$S \mapsto {\mathcal{R}}(S)$ induces an embedding of the Turing degrees into the cofinite degrees.
\item The map~$S \mapsto {\mathcal{R}}(S)$ induces an embedding of the Turing degrees into the mod-finite degrees.
\end{enumerate}
\end{proposition}

\noindent Parts~1--3 are very similar, and all follow quickly from the following lemma. Part~4 will require a more subtle construction.

\begin{lemma}\label{L:uniformrecovery}
\
\begin{enumerate}
\item Any cofinite oracle for~$S$ can uniformly compute a generic oracle for~$\widetilde{\mathcal{R}}(S)$, and any generic oracle for~$\widetilde{\mathcal{R}}(S)$ can uniformly compute a cofinite oracle for~$S$.
\item Any mod-finite oracle for~$S$ can uniformly compute a coarse oracle for~$\widetilde{\mathcal{R}}(S)$, and any coarse oracle for~$\widetilde{\mathcal{R}}(S)$ can uniformly compute a mod-finite oracle for~$S$.
\item Any (total) oracle for~$S$ can uniformly compute a cofinite oracle for~${\mathcal{R}}(S)$, and any cofinite oracle for~${\mathcal{R}}(S)$ can uniformly compute (a total oracle for)~$S$.
\end{enumerate}
\end{lemma}

\noindent Recall that a cofinite oracle is a partial oracle that halts on cofinite domain, a generic oracle is a partial oracle that halts on density 1, a mod-finite oracle is a total oracle that is correct on a cofinite set, and a coarse oracle is a total oracle that is correct on density 1.

Note that the proof is somewhat lengthy, but is quite straightforward. The only trick involved is the use of a simple voting technique to recover a mod-finite oracle from a coarse oracle.

\begin{proof}[Proof of Lemma \ref{L:uniformrecovery}]
A cofinite oracle for~$S$ can easily compute a cofinite (and hence generic) oracle for~$\widetilde{\mathcal{R}}(S)$ as follows. Each bit of~$S$ is coded into finitely many bits of~$\widetilde{\mathcal{R}}(S)$, so removing a finite amount from the oracle for~$S$ only removes a finite amount from the domain of the computation of~$\widetilde{\mathcal{R}}(S)$. Similarly, a finite number of errors in the oracle for~$S$ results in a finite number of errors in the computation of~$\widetilde{\mathcal{R}}(S)$. This produces a coarse oracle for~$\widetilde{\mathcal{R}}(S)$, since the oracle is correct on a cofinite (and hence density-1) set.

A generic oracle for~$\widetilde{\mathcal{R}}(S)$ can uniformly compute a cofinite oracle for~$S$ by searching over all of its coding locations as follows. Let~$T$ be a generic oracle for~$\widetilde{\mathcal{R}}(S)$. To compute whether~$n\in S$, for each~$m<2^n$, ask if~$2^n+m\in\widetilde{\mathcal{R}}(S)$. If~$T$ answers any of those~$2^n$-many questions, then halt and give the output that~$T$ gives. (I.e., output~$n\in S$ if and only if~$T(2^n+m)=1$, where~$m$ is the first~$m<2^n$ that is found such that~$T(2^n+m)$ halts. If no such~$m$ is found, then give no output on~$n$.)

This output is always correct, because the generic oracle never gives false outputs, so it remains to show that if the domain of the generic oracle is, in fact, density-1, then the domain of the computation is cofinite. This is because there must be some~$k$ such that 
$$\forall l>k,\ \frac{|\dom{(T\upharpoonright l)}|}{l}>\frac12.$$
When~$2^{n+1}>k$,~$T$ must halt on~$2^n+m$ for some~$m<2^n$, because otherwise~$|\dom{(T\upharpoonright 2^{n+1})}|\leq2^n$. Thus, if~$n>\log_2(k)-1$, then the computation will halt on input~$n$.

Likewise, a coarse oracle for~$\widetilde{\mathcal{R}}(S)$ can uniformly compute a mod-finite oracle for~$S$ with a voting algorithm as follows. Let~$T$ be a coarse oracle for~$\widetilde{\mathcal{R}}(S)$. To compute whether~$n\in S$, for each~$m<2^n$, ask if~$2^n+m\in\widetilde{\mathcal{R}}(S)$. When~$T$ gives~$2^{n-1}$ identical answers, then halt and give that answer as our output.

This computation always halts, because~$T$ is a total oracle, and so, when it gives~$2^n$-many answers, at least one of the two different answers must be given at least~$2^{n-1}$-many times. Thus, it remains to show that if~$T$ is correct on density-1, then the computation is correct on a cofinite set. This is because there must be some~$k$ such that 
$$\forall l>k,\ \frac{|\{
s<l:T(s)=\widetilde{\mathcal{R}}(S)(s)
\}|}{l}>\frac34.$$
When~$2^{n+1}>k$,~$T$ must be correct on~$2^n+m$ for over half of the~$m<2^n$, because otherwise~$|\{s<2^{n+1}:T(s)=\widetilde{\mathcal{R}}(S)(s)\}|\leq3(2^{n-1})$.

Finally, we show part~3. Certainly, a total oracle for~$S$ can uniformly recover all of~${\mathcal{R}}(S)$, and so, in particular, can compute a cofinite oracle for~${\mathcal{R}}(S)$.

For the converse, given a cofinite oracle for~${\mathcal{R}}(S)$, we compute~$S$ as follows. To compute whether~$n\in S$, for each odd number~$m$, we ask whether~$2^nm\in{\mathcal{R}}(S)$. When we get an answer from one of these questions, we halt and give this as our output for~$n$. The algorithm must eventually halt, because a cofinite oracle can only fail to converge at finitely many locations. Also, answer must be correct, because a cofinite oracle never gives false outputs.
\end{proof}

We now proceed to prove our proposition about embeddings.

\begin{proof}[Proof of Proposition \ref{P:embeddings}] We prove each of the parts of the proposition in turn.

\medskip
\noindent \emph{Part~1}. To show that~$S \mapsto \widetilde{\mathcal{R}}(S)$ induces an embedding of the cofinite degrees into the generic degrees, assume first that~$A\geq_\cf B$. Then,~$\widetilde{\mathcal{R}}(A)\geq_\gen\widetilde{\mathcal{R}}(B)$ by the following reduction.

Let~$S$ be a generic oracle for~$\widetilde{\mathcal{R}}(A)$. Apply Lemma \ref{L:uniformrecovery} to~$S$ to compute a cofinite oracle,~$(A)$, for~$A$. Use the reduction witnessing~$A\geq_\cf B$ to compute a cofinite oracle,~$(B)$, for~$B$. Then apply Lemma \ref{L:uniformrecovery} to~$(B)$ to compute a generic oracle for~$\widetilde{\mathcal{R}}(B)$.

The converse is similar. If~$\widetilde{\mathcal{R}}(A)\geq_\gen\widetilde{\mathcal{R}}(B)$, then~$A\geq_\cf B$ by applying the lemma, using the generic reduction, and applying the lemma again.

\medskip
\noindent \emph{Part~2.} Analogous to the proof of part~1.

\medskip
\noindent \emph{Part~3.} Analogous to the proof of part~1.

\medskip
\noindent \emph{Part~4.} We wish to show that the map~$S \mapsto {\mathcal{R}}(S)$ induces an embedding of the Turing degrees into the mod-finite degrees. First, Theorem \ref{T:imps} shows that~$\leq_\cf \implies \leq_T$, so if~${\mathcal{R}}(A)\geq_\mf{\mathcal{R}}(B)$, then~${\mathcal{R}}(A)\geq_T{\mathcal{R}}(B)$. In this case,~$A\geq_TB$, because~$A\equiv_T{\mathcal{R}}(A)\geq_T{\mathcal{R}}(B)\equiv_TB$.

For the converse, assume~$A\geq_TB$. Then, in particular,~$A\geq_T{\mathcal{R}}(B)$. Assume that this is witnessed by~$\Phi$, so that~$\Phi^A=\mathcal{R}(B)$. We will use this~$\Phi$ to provide a computation of~$\mathcal{R}(B)$ from~$\mathcal{R}(A)$ for which a finite error in the oracle will result in a finite error in the output.

For any number,~$k$, and real,~$X$, let~$X_k=\{n:2^n(2k+1)\in X\}$. So~$(\mathcal{R}(A))_k=A$. The idea is that~$\mathcal{R}(A)$ has~$\omega$-many distinct copies of~$A$ in it, and~$(\mathcal{R}(A))_k$ is the~$k$th copy. The most important fact about this is that if~$C$ is any mod-finite oracle for~$\mathcal{R}(A)$, then for all but finitely many~$k$,~$C_k=A$.

So now, let~$C$ be a mod-finite oracle for~$\mathcal{R}(A)$. To compute whether~$k\in\mathcal{R}(B)$, we search for some~$l>k$ such that~$\Phi^{C_l}(k)\downarrow$. When we find such an~$l$, we halt, and give~$\Phi^{C_l}(k)$ as our output.

To finish the proof, we must show that for every~$k$, we halt, and that for all but finitely many~$k$, we halt and give the correct answer. We halt for every~$k$ because for every~$k$,~$\Phi^{A}(k)$ halts, and there is some~$l>k$ such that~$C_l=A$, so at some point, we will find that computation and halt (if we have not halted before). We only give finitely many incorrect answers because there are only finitely many~$n$ such that~$C_n\neq A$, so if~$k$ is larger than all of those~$n$, then for every~$l>k$,~$C_l=k$ and so, in particular,~$\Phi^{C_l}(k)=\Phi^A(k)=\mathcal{R}(B)(k)$. Thus, for sufficiently large~$k$, the first found output will be the correct output.
\end{proof}

Our final results in this section establish that infinite-information reducibility behaves in many ways \emph{opposite} to the rest (a fact we shall find further evidence for in Section \ref{S:inf}). Part~2 below, that the Turing join is the infinite-information meet, is false of most reducibilities studied in the literature, including all the other ones we are considering here, as we show subsequently in Proposition \ref{P:joinsarejoins}.

\begin{proposition}\label{P:iibasics}
\
\begin{enumerate}
\item For all sets~$A$ and~$B$, if~$B \leq_1 A$ then~$A \leq_\ii B$.
\item The (Turing) join of two sets~$A$ and~$B$ is their meet under~$\leq_\ii$.
\end{enumerate}
\end{proposition}

\begin{proof}
For part~1, suppose~$f$ witnesses the~$1$-reduction of~$B$ to~$A$. Let~$\Phi$ be the corresponding Turing functional, and define~$\Psi$ by letting
\[
\Psi^{(S)}(m) = 
\begin{cases}
(S)(f^{-1}(m)) & \textrm{if } m \in \ran(f),\\
\uparrow & \textrm{otherwise}.
\end{cases}
\]
for all partial oracles~$(S)$ and numbers~$m$. In particular, if~$(B)$ is a partial oracle for~$B$ and~$m = f(n)$ for some~$n$ in the domain of~$(B)$ then
\[
\Psi^{(B)}(m) = B(n) = A(f(n)) = A(m).
\]
It follows that if~$(B)$ has infinite domain then~$\Psi^B$ is a partial oracle for~$A$ with infinite domain.

For part~2, we have that~$A \oplus B \leq_\ii A,B$, since we can uniformly convert infinitely many bits of~$A$ or of~$B$ into infinitely many bits of~$A \oplus B$. Furthermore, if~$C \leq_\ii A$ via~$\Phi$ and~$C \leq_\ii B$ via~$\Psi$, then~$C \leq_\ii A \oplus B$ via the functional, given oracle~$S$ and input~$m$, runs~$\Phi^{\{n : 2n \in S\}}(m)$ and~$\Psi^{\{n : 2n+1 \in S\}}(m)$, returning the output of whichever of these that happens to halt first.
\end{proof}

\begin{proposition}\label{P:joinsarejoins}
The (Turing) join of two sets~$A$ and~$B$ is also their join under any of our reducibilities other than~$\leq_\ii$.
\end{proposition}

\begin{proof}
The result is clear for~$\leq_\ubfb$, and is easily checked for the remaining reducibilities. For instance, to show that~$A \oplus B$ is the join of~$A$ and~$B$ under~$\leq_\mr$, note that if~$C = C_0 \oplus C_1$ agrees with~$A \oplus B$ on a computable domain~$D = D_0 \oplus D_1$ then~$C_0$ and~$C_1$ agree with~$A$ and~$B$ on the computable domains~$D_0$ and~$D_1$, respectively. Hence,~$A,B \leq_\mr A \oplus B$. On the other hand, if~$A,B \leq_\mr S$, say via reductions~$\Phi$ and~$\Psi$, then~$\Phi \oplus \Psi$ witnesses that~$A \oplus B \leq_\mr S$.
\end{proof}

\section{Non-implications}\label{S:nonimps}

We divide the task of proving part~2 of Theorem \ref{T:imps}, that no further implications hold between our reducibilities other than the ones presented there, among the following subsections.

\subsection{Infinite-information, generic, and coarse reducibilities}

In this section, we show that~$\leq_\ii$,~$\leq_\gen$, and~$\leq_\cor$ are independent of each other and our other reducibilities in terms of implication. For~$\leq_\ii$, this follows readily from our work at the end of the previous section.

\begin{proposition}
If~$\leq_*$ is~$\leq_T$,~$\leq_1$, or any of our reducibilities other than~$\leq_\ii$, then~$\leq_* \nimplies \leq_\ii$ and~$\leq_\ii \nimplies \leq_*$. 
\end{proposition}

\begin{proof}
For all sets~$A$, we have
\[
A \leq_* A \oplus \emptyset \leq_\ii \emptyset,
\]
by Propositions \ref{P:joinsarejoins} and \ref{P:iibasics}. Thus, if~$\leq_* \implies \leq_\ii$ we have~$A \leq_\ii \emptyset$, and~$\leq_\ii \implies \leq_*$ we have~$A \leq_* \emptyset$. Thus, to disprove either implication, let~$A$ be any weakly~$1$-generic or~$1$-random set, so that by Proposition \ref{P:randomgenerics}, we have~$A \nleq_\ii \emptyset$ and~$A \nleq_* \emptyset$.
\end{proof}

We now turn to~$\leq_g$ and~$\leq_\cor$. The incomparability of these two with one another follows from the work of Jockusch and Schupp~\cite{JS-2012}.

\begin{proposition}
\
\begin{enumerate}
\item There exist sets~$A$ and~$B$ such that~$B \leq_\gen A$ but~$B \nleq_\cor A$.
\item There exist sets~$A$ and~$B$ such that~$B \leq_\cor A$ but~$B \nleq_\gen A$.
\end{enumerate}
\end{proposition}

\begin{proof}
Jockusch and Schupp~\cite[Theorem~2.26 and Proposition~2.15]{JS-2012} exhibited a set that is generically computable but not coarsely computable, and a set that is coarsely computable but not generically computable. We thus take~$A$ to be~$\emptyset$, and~$B$ to be the set from the relevant Jockusch-Schupp construction.
\end{proof}

Finally, we establish that~$\leq_\gen$ and~$\leq_\cor$ are incomparable with all of the other reducibilities mentioned in Theorem \ref{T:imps}. Our proof below actually works for both the uniform and nonuniform versions of these reducibilities, as discussed in Section~\ref{S:intro}.

\begin{proposition}
\
\begin{enumerate}
\item There exist sets~$A$ and~$B$ such that~$B \leq_1 A$ but~$B \nleq_\cor A$.
\item There exist sets~$A$ and~$B$ such that~$B \leq_\cor A$ but~$B \nleq_T A$.
\end{enumerate}
Hence if~$\leq_*$ is~$\leq_T$,~$\leq_1$, or any of our reducibilities other than~$\leq_\cor$, then~$\leq_* \nimplies \leq_\cor$ and~$\leq_\cor \nimplies \leq_*$. The same is true for~$\leq_\gen$ in place of~$\leq_\cor$. 
\end{proposition}

\begin{proof}
For part~1, fix any~$B \nleq_\cor \emptyset$ and let~$A = \{2^n : n \in B\}$. Then obviously~$B \leq_1 A$ but~$A \leq_\cor \emptyset$, so~$B$ is not coarsely reducible to~$A$. For part~2, let~$C$ be any non-computable set,~$B = \{2^n : n \in C\}$, and~$A = \emptyset$. Then~$B \leq_\cor A$, but~$B \nleq_T A$ since~$B \equiv_T C$. The proofs for~$\leq_\gen$ are similar.
\end{proof}

\subsection{Mod-recursive reducibility}

We begin in this section with a straightforward result that shows that the implication from~$\leq_1$ to~$\leq_\mr$ is strict.

\begin{proposition}\label{P:mrnimp1}
There exist sets~$A$ and~$B$ such that~$B \leq_\mr A$ but~$B \nleq_1 A$.
\end{proposition}

\begin{proof}
Fix any set~$A$ such that~$\overline{A} \nleq_1 A$, and let~$B = \overline{A}$. Then observe that every set is mod-recursive reducible to its complement via the functional~$\Phi$ defined by~$\Phi^C(n) = 1 - C(n)$ for all sets~$C$ and numbers~$n$. Indeed, we have that~$\{n : \Phi^C(n) = \overline{A}(n)\} = \{n : 1- C(n) = 1 - A(n)\} = \{n : C(n) = A(n)\}$.
\end{proof}

The next proposition establishes that~$\leq_\mr$ is not implied by~$\leq_\mf$, or in fact, any of our other reducibilities, since all of these are implied by~$\leq_\mf$. Thus, no other implications to~$\leq_\mr$ can be added to Theorem \ref{T:imps}. The  argument below follows a basic outline that will be used for several others in this section. We have a definition~$\Delta$ of a set~$B$ from a set~$A$ in mind that happens to be a reduction of a particular kind (below, a mod-finite reduction) but not of another (below, a mod-recursive reduction). Then, exploiting the fact that if~$\Phi$ witnesses a reduction of the latter kind it must be that~$\Phi^A = B$, we force~$\Phi$ to either equal~$\Delta$ or to make a mistake, and thus, in either case, to not witness the latter reduction after all.

\begin{proposition}
There exist sets~$A$ and~$B$ such that~$B \leq_\mf A$ but~$B \nleq_\mr A$.
\end{proposition}

\begin{proof}
We construct a set~$U$ and define~$B$ by~$B(n) = U(2n)U(2n+1)$. We then let~$A = \mathcal{R}(U)$. Thus,~$B \leq_\mf A$ and~$A \leq_\mr U$, so it suffices to show that~$B \nleq_\mr U$. To this end, we define for each~$s \in \omega$ the set
\[
U_s= U \res s \cup \{2n \geq s : 2n \in U\} \cup \{2n+1 \geq s : 2n+1 \notin U\},
\]
and satisfy the following requirement for each~$e \in \omega$:
\[
\begin{array}{lll}
R_e & : & \Phi^U_e \neq B \textrm{ or for some~$s$, either } \Phi_e^{U_s} \textrm{ is not total or }\\
& &  \{n \in \omega: \Phi^{U_s}_e(n) = B(n)\} \textrm{ is not computable.}
\end{array}
\]
Since~$\{n \in \omega : U_s(n) = U(n)\}$ is obviously computable for each~$s$, this will ensure that~$\Phi_e$ does not witness a mod-recursive reduction.

\medskip
\noindent \emph{Construction.} We let~$U = \bigcup_s \sigma_s$, where~$\sigma_0 \preceq \sigma_1 \preceq \cdots$ are finite strings obtained as follows. Let~$\sigma_0 = \emptyset$, and suppose some~$\sigma_e$ is given.

\medskip
\noindent \emph{Step 1.} Ask if there exists a~$\sigma \succeq \sigma_e$ such that~$\Phi^{\sigma}_e(n) \downarrow \neq \sigma(2n)\sigma(2n+1)$ for some~$n$ with~$2n+1 < |\sigma|$. If so, fix the least such~$\sigma$, and otherwise let~$\sigma = \sigma_e$.

\medskip
\noindent \emph{Step 2.} Ask if there is an~$n$ with~$2n \geq |\sigma|$ such that~$\Phi_e(2n) \downarrow$. If so, let~$\sigma_{e+1}$ be the least extension of~$\sigma$ with~$\sigma_{e+1}(2n) \neq \Phi_e(2n)$, and otherwise let~$\sigma_{e+1} = \sigma$.

\medskip
\noindent \emph{End construction.}

\medskip
\noindent \emph{Verification.} Fix~$e$, and suppose~$\Phi^U_e = B$. Then it must be that in the construction, we could not find a~$\sigma \succeq \sigma_e$ satisfying the question asked at Step 1. In other words, any set~$S$ extending~$\sigma_e$ satisfies~$\Phi_e^S(n) \simeq S(2n)S(2n+1)$. In particular, this holds for~$S = U_s$, where~$s = |\sigma_e|$. Assuming~$\Phi_e^{U_s}$ is total, this means that for all~$n$ with~$2n+1 \geq s$,
\[
\Phi^{U_s}_e(n) = U_s(2n)U_s(2n+1) = U(2n)[1-U(2n+1)],
\]
which can only equal~$B(n) = U(2n)U(2n+1)$ provided~$U(2n) = 0$. Thus,~$\{n \in \omega : \Phi^{U_s}_e(n) = B(n)\}$ is computable if and only if~$\{2n \in \omega : 2n \in U\}$ is, but we ensured that the latter set is not computable under Step 2 of the construction.
\end{proof}

It remains to show that no further implications from~$\leq_\mr$ are possible either. This follows from the following result, again with a relatively straightforward proof, which completes our analysis of mod-recursive reducibility.

\begin{proposition}\label{P:mrnimpcf}
There exists sets~$A$ and~$B$ such that~$B \leq_\mr A$ but~$B \nleq_\cf A$.
\end{proposition}

\begin{proof}
Consider any set~$A$ that is not autoreducible, meaning there is no functional~$\Phi$ such that~$A = \Phi^A$ and the computation of~$\Phi^A(n)$ does not query~$A(n)$. Let~$B = \mathcal{R}(A)$. Since~$A$ cannot be uniformly recovered from cofinitely many of its own bits, but can be recovered from cofinitely many of the bits of~$B$, it follows that~$B \nleq_\cf A$. However,~$B \leq_\mr A$, as witnessed by the functional~$\Psi$ taking a set~$C$ to~$\mathcal{R}(C)$, since then
\[
\{ k : \Psi^C(k) = B(k) \} = \{ 2^n m : C(n) = A(n),~m \textrm{ odd}\} \leq_T \{n \in \omega: C(n) = A(n)\}.\qedhere
\]
\end{proof}

\subsection{Mod-finite, uniform-bounded-from-below, and cofinite reducibilities}

In this section, we focus on the remaining open implications in Theorem \ref{T:imps}, which are now just the reversals of those in the chain~$\leq_1 \implies \leq_\mf \implies \leq_\ubfb \implies \leq_\cf \implies \leq_T$. The following sequence of propositions establishes that none of these can be reversed.

\begin{proposition}
There exist sets~$A$ and~$B$ such that~$B \leq_\mf A$ but~$B \nleq_1 A$.
\end{proposition}

\begin{proof}
This can be proved identically to Proposition \ref{P:mrnimp1}. Alternatively, this can be observed from the facts, proved above, that~$\leq_1 \implies \leq_\mf$ and~$\leq_\mr \not\implies \leq_\mf$.
\end{proof}

\begin{proposition}
There exist~$A$ and~$B$ such that~$B \leq_\ubfb A$ but~$B \nleq_\mf A$.
\end{proposition}

\begin{proof}
We construct a set~$A$ by finite approximation with the property that for each~$k \in \omega$ there is exactly one~$n$ such that~$\seq{n,k} \in A$. In other words, each column of~$A$ has exactly one element. The set~$B$ is then defined from~$A$ as the set of indices of columns containing an odd element, that is,
\[
B = \{n \in \omega: \seq{n,k} \in A \textrm{ for some odd } k\}.
\]
Clearly, this makes~$B \leq_\ubfb A$. We now build~$A$ so that~$B \nleq_\mf A$. Intuitively, by removing the single element of some column of~$A$, and hence making only a finite change, we shall be able to cause a purported computation of~$B$ to not be total. We pass to the details.

\medskip
\noindent \emph{Construction.} We let~$A = \bigcup_e \sigma_e$, where~$\sigma_0 \preceq \sigma_1 \preceq \cdots$ are finite strings obtained as follows. Let~$\sigma_0 = \emptyset$. Call a string~$\sigma \in 2^{<\omega}$ \emph{valid} if for each~$k$ there is at most one~$n$ with~$\seq{n,k} < |\sigma|$ and~$\sigma(\seq{n,k}) = 1$. Now assume inductively that we have defined~$\sigma_e$ for some~$e$, and that it is valid. 

Ask if there is a valid extension~$\sigma$ of~$\sigma_s$ such that for some~$n,k$ with~$\seq{n,k} < |\sigma|$ and~$\sigma(\seq{n,k}) = 1$, either~$k$ is odd and~$\Phi_e^\sigma(n) \downarrow = 0$, or~$k$ is even and~$\Phi_e^\sigma(n) \downarrow = 1$. If so, let~$\sigma_{e+1}$ be the least such~$\sigma$. Otherwise, define~$n_e$ as the least number for which there is no~$k$ with~$\sigma_e{\seq{n_e,k}} = 1$, and define~$k_e$ as the least number with~$\seq{n_e,k_e} \geq |\sigma_e|$. Then let~$\sigma_{e+1}$ be the least valid extension of~$\sigma_e$ with~$\sigma_e(\seq{n_e,k_e}) = 1$.

\medskip
\noindent \emph{Verification.} We now fix~$e \in \omega$ and verify that~$\Phi_e$ does not witness a mod-finite reduction of~$B$ to~$A$. We may assume~$B = \Phi_e^A$, since otherwise we are done. Hence, during the construction, we must not have been able to diagonalize on~$\Phi_e$ against our method of defining~$B$ from~$A$. In other words, for any set~$C$ extending~$\sigma_e$ that contains at most one element in each column, if~$\Phi^C_e$ is total then it equals~$\{n \in \omega: \seq{n,k} \in C \textrm{ for some odd } k\}$. Furthermore,~$n_e$ and~$k_e$ must have been defined and~$\seq{n_e,k_e}$ put in~$A$. In particular, for the set~$C = A - \{ \seq{n_e,k_e}\}$, if~$\Phi^C_e(\seq{n_e,k_e})$ converged it would have to equal~$0$. But any initial segment of~$C$ witnessing this computation could then be continued by the addition of~$\seq{n_e,k}$ for some odd~$k$, which would be a contradiction. Hence,~$\Phi^C_e(\seq{n_e,k_e})$ must be undefined, and~$\Phi^C_e$ not total. Since~$C =^* A$, this means~$\Phi_e$ is not a mod-finite reduction.
\end{proof}

The coding techniques in our next proposition are more sophisticated than we have used so far. The argument below is similar to that used by Igusa \cite{???} to show that as a relation on sets, generic reducibility is~$\mathbf{\Pi}^1_1$-complete.

\begin{proposition}
There exist~$A$ and~$B$ such that~$B \leq_\cf A$ but~$B \nleq_\ubfb A$.
\end{proposition}

\begin{proof}
We shall construct the set~$A$, and define~$B$ from it. Thus, before proceeding to the construction, we describe this definition. We partition~$\omega$ as
\[
\{a_{i,0} : i \in \omega \} \cup \{a_{i,1} : i \in \omega \} \cup \{a_{i,j,k,0} : i,j,k \in \omega \} \cup \{a_{i,j,k,1} : i,j,k \in \omega \},
\]
and build~$A$ to ensure the following properties hold for every~$i$:
\begin{enumerate}
\item~$A$ contains at most one of~$a_{i,0}$ and~$a_{i,1}$;
\item if~$a_{i,0} \in A$ then~$a_{i,j,k,1} \notin A$ for all~$j, k$;
\item if~$a_{i,1} \in A$ then~$a_{i,j,k,0} \notin A$ for all~$j, k$;
\item if~$a_{i,j,k,0} \in A$ for some~$j, k$, then~$a_{i,j,k',1} \notin A$ for all~$k'$;
\item if~$a_{i,j,k,1} \in A$ for some~$j, k$, then~$a_{i,j,k',0} \notin A$ for all~$k'$;
\item for almost every~$j$ there is a~$k$ with~$a_{i,j,k,0} \in A$ or~$a_{i,j,k,1} \in A$;
\item if~$a_{i,0}, \; a_{i,1} \notin A$, then for every~$j$ there is a~$k$ with~$a_{i,j,k,0} \in A$ or~$a_{i,j,k,1} \in A$.
\end{enumerate}
We regard~$B$ as a set of pairs, and define it from~$A$ by the rules below. For all~$i$:
\begin{itemize}
\item if~$a_{i,0} \in A$ then~$\seq{i,j} \notin B$ for all~$j$;
\item if~$a_{i,1} \in A$ then~$\seq{i,j} \in B$ for all~$j$;
\item for every~$j$, if there is a~$k$ with~$a_{i,j,k,0} \in A$, then~$\seq{i,j} \notin B$;
\item for every~$j$, if there is a~$k$ with~$a_{i,j,k,1} \in A$, then~$\seq{i,j} \notin B$.
\end{itemize}
Intuitively, we are think of the values of the bits of~$B$ as being \emph{deduced} from the values of certain corresponding bits of~$A$. Namely, for each~$i$, the bits~$a_{i,0}$ and~$a_{i,1}$ represent a kind of \emph{master deduction procedure} for determining the values of the collection of bits~$\seq{i,0}, \seq{i,1}, \ldots$, while for all~$i$ and~$j$, each of the bits~$a_{i,j,k,0}$ and~$a_{i,j,k,1}$ represents an \emph{individual deduction procedure} for determining the value of the single bit~$\seq{i,j}$. Properties~1--5 above ensure that none of the deduction procedures contradict each other, and hence that the definition of~$B$ is consistent. Property 6 will ensure that the intended reduction is a cofinite reduction. Property 7 ensures that the above definition determines~$B$ completely.

In fact, we can turn this definition into a cofinite reduction. Let~$\Phi$ be the functional which with (possibly partial) oracle~$C$ and input~$\seq{i,j}$ queries the bits~$a_{i,0}$ and~$a_{i,1}$ of~$C$, whilst simultaneously querying the bits~$a_{i,j,k,0}$ and~$a_{i,j,k,1}$ for all~$k$ one by one. If and when it finds some such bit to be in~$C$, it halts and outputs an answer in accordance with our definition of~$B$ above. Properties~2,~3,~and~7 ensure that~$\Phi^A = B$. The remaining properties, together with the fact that we are querying (bits representing) master and individual deduction procedures simultaneously, ensure that if~$(A)$ is a partial oracle for~$A$ with cofinite domain then~$\Phi^{(A)} \simeq B$ with cofinite domain. Indeed, say~$m$ is such that~$(A)(n) \downarrow = A(n)$ for all~$n > m$. For cofinitely many triples~$i$,~$j$, and~$k$, each of the numbers~$a_{i,0}$,~$a_{i,1}$,~$a_{i,j,k,0}$, and~$a_{i,j,k,1}$ is greater than~$m$, so for cofinitely many pairs~$i$ and~$j$ the reduction~$\Phi^{(A)}$ will halt on~$\seq{i,j}$. Thus,~$B \leq_\cf A$.

We now turn to the construction of~$A$.

\medskip
\noindent \emph{Construction.} We must ensure that for all~$e$, the functional~$\Phi_e$ does not witness a uniform-bounded-from-below reduction of~$B$ to~$A$. The idea is as follows. First, we reduce to the case where~$\Phi_e$ computes~$B$ from~$A$ essentially via the definition of~$B$ given above, that is, by looking for master deductions and individual deductions. Otherwise, we argue that~$\Phi^A_e \neq B$. Then, we fix a number~$i$ and add one of the master deduction~$a_{i,0}$ or~$a_{i,1}$ to~$A$. So one way for~$\Phi^A_e$ to correctly compute the value of~$\seq{i,j}$ for a given~$j$ is to query this master deduction, but of course if it does this for infinitely many~$j$ then its use will not be bounded from below. Our strategy, then, is to pick a new~$j$ every time one of the master deductions is queried, and keep all of its individual deductions~$a_{i,j,k,0}$ and~$a_{i,j,k,1}$ out of~$A$ until one of the master deductions is queried again. So if indeed~$\Phi^A_e$ eventually stops querying the master deductions, meaning that we settle on a permanent value of~$j$, it will not be able to converge on~$\seq{i,j}$. At the same time, we permanently restrain individual deductions for at most one~$j$, thereby ensuring that~$A$ satisfies the properties above.

Formally, we proceed by stages. At each stage~$s$, we shall have a finite set~$I_s$ of numbers smaller than~$s$, and for each~$e \in I_s$, unique numbers~$i_e$ and~$j_{e,s}$. For~$e \in I_s$, we say the bits of the form~$a_{i_e,j_{e,s},k,0}$ and~$a_{i_e,j_{e,s},k,1}$ are \emph{$(e,s)$-forbidden}. A string~$\sigma$ will be called \emph{valid} at stage~$s$ if:
\begin{itemize}
\item the range of~$\sigma$ does not violate properties 1--5 above if~$\sigma$ is regarded as an initial segment of~$A$;
\item if the range of~$\sigma$ contains an~$(e,s)$-forbidden number~$n$ for some~$e \in I_s$, then~$\Phi^{\sigma \res n}_e(\seq{i_e,j_{e,s}}) \downarrow$ and the computation queries either~$a_{i_e,0}$ or~$a_{i_e,1}$.
\end{itemize}
We obtain our set~$A$ as a union~$\bigcup_s \sigma_s$ of valid strings.

Initially, let~$\sigma_0 = \emptyset$, and suppose inductively that for some~$s \geq 0$, we are given~$\sigma_s$,~$I_s$, and for each~$e \in I_s$, the numbers~$i_e$ and~$j_{e,s}$. By replacing~$\sigma_s$ with a valid extension if necessary, we may assume that for each~$e$ in the finite set~$I_s$, either no valid extension of~$\sigma_s$ has an~$(e,s)$-forbidden element in its range, or~$\sigma_s$ does so already. Let~$i_s$ be least so that for all~$i \geq i_s$ and all~$j$ and~$k$, none of the numbers~$a_{i,0}$,~$a_{i,1}$,~$a_{i,j,k,0}$ or~$a_{i,j,k,1}$ are smaller than~$|\sigma_s|$. In other words, we have not yet determined the values of these bits in~$A$.

We begin by extending~$\sigma$ to a string~$\widetilde{\sigma}$, considering two cases.

\medskip
\noindent \emph{Case~1:} there is a valid~$\sigma \succeq \sigma_s$ such that for some~$j$ and some~$b \in \{0,1\}$, we have that~$\Phi^\sigma_s(\seq{i_s,j}) \downarrow = b$, and if the computation queries~$a_{i_s,b}$ or~$a_{i_s,j,k,b}$ for some~$k$ then this bit is not in~$A$. Let~$I_{s+1} = I_s$, and fix some such~$\sigma$,~$j$, and~$b$. Let~$\widetilde{\sigma}$ be obtained by deleting any element of the form~$a_{i_s,b}$ or~$a_{i_s,j,k,b}$ from the range of~$\sigma$, and appending some~$a_{i_s,j,k,1-b} \geq |\sigma|$. Then~$\widetilde{\sigma}$ is a valid extension of~$\sigma_s$ by choice of~$i_s$, and~$\Phi^{\widetilde{\sigma}}_s(\seq{i_s,j})$ converges and agrees with~$\Phi^\sigma_s(\seq{i_s,j})$ by choice of~$\sigma$. Note that by construction,~$\Phi^{\widetilde{\sigma}}_s$ disagrees with~$B$ on~$\seq{i_s,j}$.

\medskip
\noindent \emph{Case~2:} otherwise. Let~$I_{s+1} = I_s \cup \{s\}$, and define~$j_{s,s+1} = 0$. Let~$\widetilde{\sigma}$ be the least extension of~$\sigma_s$ with either~$a_{i_s,0}$ or~$a_{i_s,1}$ in its range.

\medskip
Next, consider any~$e \in I_s$ such that~$\widetilde{\sigma}$ contains an~$(e,s)$-forbidden element in its range. Choose the least~$j > j_{e,s}$ such that no element of the form~$a_{i_e,j,k,0}$ or~$a_{i_e,j,k,1}$ is smaller than~$|\widetilde{\sigma}|$, and let~$j_{e,s+1} = j$. Thus, no~$(e,s)$-forbidden element is~$(e,s+1)$-forbidden.

Finally, let~$\sigma_{s+1}$ be the least valid extension of~$\widetilde{\sigma}$ such that for all~$e \leq s$ and all~$\seq{i,j} \leq s$ different from~$\seq{i_e,j_{e,s+1}}$, the range of~$\sigma_{s+1}$ contains~$a_{i,j,k,0}$ or~$a_{i,j,k,1}$ for some~$k$.

\medskip
\noindent \emph{Verification.} By definition of validity, it follows that~$A = \bigcup_s \sigma_s$ satisfies properties~1--5 above. The last step at stage~$s+1$ of the construction ensures that if~$\seq{i,j} \leq s$ and~$j \neq j_{e,s}$ for any~$e \in I_s$ then some element of the form~$a_{i,j,k,0}$ or~$a_{i,j,k,1}$ gets added to~$A$. It follows that if~$i \neq i_e$ for any~$e$, or if~$i = i_e$ and~$\lim_s j_{e,s}$ does not exist, then this happens for all~$j$. But if~$i = i_e$ and then~$a_{i,0}$ or~$a_{i,1}$ is added to~$A$, so if the~$j_{e,s}$ comes to a limit~$j_e$ then the above happens for all~$j \neq j_e$. Thus, properties~6 and~7 hold as well.

To finish the proof, fix any functional~$\Phi_e$. If Case~1 of the construction applies at stage~$e+1$ of the construction, then we ensure that~$\Phi_e^A \neq B$. If Case~2 applies, then~$\Phi^A_e(\seq{i_e,j})$ can only converge to some~$b \in \{0,1\}$ if it queries either~$a_{i_e,b}$ or~$a_{i_e,j,k,b}$ for some~$k$, and finds this bit present in~$A$. In this case, there are two possibilities. If the~$j_{e,s}$ do not come to a limit during the construction, then one of~$a_{i_e,0}$ or~$a_{i_e,1}$ must be queried infinitely often, so~$\Phi_e^A$ is not a uniform-bounded-from-below reduction. If, on the other hand,~$j_{e,s}$ do come to some limit~$j_e$, then neither~$a_{i_e,0}$ nor~$a_{i_e,1}$ can be queried after~$j_{e,s}$ assumes its final value,~$j_e$. But then the construction ensures that neither~$a_{i_e,j_e,k,0}$ nor~$a_{i_e,j_e,k,1}$ are in~$A$, so~$\Phi^A_e(\seq{i_e,j_e})$ cannot converge. We conclude that in any case,~$\Phi_e^A$ is not a uniform-bounded-from-below computation of~$B$.
\end{proof}

We finish this section with the following straightforward result.

\begin{proposition}
There exist~$A$ and~$B$ such that~$B \leq_T A$ but~$B \nleq_\cf A$.
\end{proposition}

\begin{proof}
This follows from Proposition \ref{P:mrnimpcf} since~$\leq_\mr$ implies~$\leq_T$ but not~$\leq_\cf$.
\end{proof}

\section{A maximal pair of infinite-information degrees}\label{S:inf}

In this final section, we take a further look at~$\leq_\ii$, which turns out to be arguably the most curious of our reducibilities.  As we already saw in Proposition~\ref{P:iibasics}, the~$1$-degrees embed backwards into the infinite-information degrees, and the usual join operator provides a meet in the infinite-information degrees.

In fact, it turns out that unlike most reducibilities studied in the literature,~$\leq_\ii$ simply does not admit a join operator at all. In other words, the power set of~$\omega$ ordered by~$\leq_\ii$ is not an upper semi-lattice.

\begin{theorem}\label{T:maxpair}
There exist sets~$B_0$ and~$B_1$ for which there is no set~$A$ such that~$B_0 \leq_\ii A$ and~$B_1 \leq_\ii A$.
\end{theorem}

\noindent To coin a phrase based on standard terminology, the infinite-information degrees thus admit \emph{maximal pairs}. Just as in the Turing degrees there are minimal pairs, which are so simple that they have no common information content, maximal pairs here are so complex that no set possesses common infinite-information content about them.

The complexity alluded to above will be made apparent by the proof of Theorem \ref{T:maxpair}. This consists of two technical results, Lemmas \ref{L:maxpair1} and \ref{L:maxpair2} below, employing the following definition.

\begin{definition}\label{D:nleqii}
For sets~$A$ and~$B$, we write~$B \from_\ii A$ if every~$C \geq_T A$ can produce an infinite-information computation of~$B$, i.e., if there is a Turing functional~$\Phi$ with~$\Phi^C \simeq B$ with infinite domain.
\end{definition}

\noindent It is not difficult to show that as a relation on sets,~$\from_\ii$ need not be transitive. However, it does enjoy the following limited form of transitivity with respect to~$\leq_\ii$: if~$C \leq_\ii B \from_\ii A$ then~$C \from_\ii A$. Indeed, if~$\Phi$ witnesses that~$C \leq_\ii B$ and~$\Psi$ witnesses that~$B \from_\ii A$, then~$\Phi \circ \Psi$ will witness that~$C \from_\ii A$. We shall make use of this fact in the sequel. Beyond this,~$\leq_\ii$ and~$\from_\ii$ appear ostensibly quite similar. But in fact the two differ in another key way. Namely, if~$B \from_\ii A$ and~$C$ is any set then necessarily~$B \from_\ii A \oplus C$. As Proposition \ref{P:iibasics} shows, this is not true of~$\leq_\ii$.

The first lemma below is quite straightforward. The second, by contrast, turns out to require a substantial amount of work, and we defer its proof for the time being.

\begin{lemma}\label{L:maxpair1}
There exist sets~$B_0$ and~$B_1$ such that for every~$\alpha<\omega_1^{CK}$, we have that~$B_0 \nleftarrow_\ii B_1 \oplus \emptyset^{(\alpha)}$ and~$B_1 \nleftarrow_\ii B_0 \oplus \emptyset^{(\alpha)}$.
\end{lemma}

\begin{proof}

We build~$B_0$ and~$B_1$ with~$B_0 \nleftarrow_\ii B_1 \oplus \mathcal{O}$ and~$B_1 \nleftarrow_\ii B_0 \oplus \mathcal{O}$. This will satisfy the lemma because, for every~$\alpha<\omega_1^{CK}$,~$\emptyset^{(\alpha)}<_T\mathcal{O}$. We use finite approximation to build~$B_0$ and~$B_1$.

Let~$\sigma_0 = \tau_0 = \emptyset$, and assume that for some~$s \geq 0$,~$\sigma_s$ and~$\tau_s$ are defined. If~$s = 2e$, we ask if there is any~$n \geq |\sigma_s|$ and~$\tau \succeq \tau_s$ such that
\[
\Phi_e^{\tau \oplus (\mathcal{O} \res |\tau|)}(n) \downarrow = b
\]
for some~$b \in \{0,1\}$, and if so we let~$\tau_{s+1}$ be the least such~$\tau$, and let~$\sigma_{s+1}$ be the least~$\sigma \succeq \sigma_s$ with~$n < |\sigma|$ and~$\sigma(n) \neq b$. This ensures that either~$\Phi_e^{B_1 \oplus \mathcal{O}}$ has finite domain or it disagrees with~$B_0$. If~$s = 2e+1$, we analogously diagonalize to ensure that either~$\Phi_e^{B_0 \oplus \mathcal{O}}$ has finite domain or it disagrees with~$B_1$.
\end{proof}

\begin{lemma}\label{L:maxpair2}
For all sets~$A$ and~$B$, if~$B \leq_\ii A$ then either there exists an~$\alpha$ such that~$B \from_\ii \emptyset^{(\alpha)}$ or there exists an~$\alpha$ such that~$A \from_\ii B \oplus \emptyset^{(\alpha)}$.
\end{lemma}

We now pause to indicate how Theorem \ref{T:maxpair} follows.

\begin{proof}[Proof of Theorem \ref{T:maxpair}]
Fix~$B_0$ and~$B_1$ as in the Lemma \ref{L:maxpair1}, and let~$A$ be any set with~$B_0 \leq_\ii A$. By Lemma \ref{L:maxpair2}, fic some~$\alpha$ such that either~$B_0 \from_\ii \emptyset^{(\alpha)}$ or~$A \from_\ii B_0 \oplus \emptyset^{(\alpha)}$. The former case cannot hold, since if it did we would have~$B_0 \from_\ii \emptyset^{(\alpha)} \oplus B_1$ by the remark after Definition \ref{D:nleqii}, a contradiction. But by the same remark, if we had~$B_1 \leq_\ii A \from_\ii B_0 \oplus \emptyset^{(\alpha)}$ we would have~$B_1 \from_\ii B_0 \oplus \emptyset^{(\alpha)}$, again a contradiction. Thus,~$B_1 \nleq_\ii A$, and the theorem is proved.
\end{proof}

The rest of the section is now dedicated to the proof of Lemma \ref{L:maxpair2}, relying in turn, on the three technical lemmas below. Call a partial oracle \emph{finite} if its domain is finite. Given finite partial oracles~$\sigma$ and~$\tau$, we say~$\tau$ is a \emph{$1$-extension} of~$\sigma$, and write~$\tau \succ_1 \sigma$, if~$\tau$ extends~$\sigma$ and is defined on precisely one more element than~$\sigma$ is.

\begin{definition}
Let~$\sigma$ be a partial oracle,~$n \in \omega$, and~$i \in \{0,1\}$.
\begin{enumerate}
\item For~$\alpha$ an ordinal,~$\sigma$ a finite partial oracle,~$n$ an integer, and~$i \in \{0,1\}$, we say \emph{$\sigma$~$\alpha$-deduces that~$\Phi(n)=i$} if~$\Phi^\sigma(n) \downarrow =i$, or if there exist infinitely many~$\tau\succ_1\sigma$, each of which~$\beta$-deduces that~$\Phi(n)=i$ for some~$\beta < \alpha$.
\item We say \emph{$\sigma$ deduces that~$\Phi(n)=i$} if there is an ordinal~$\alpha$ such that~$\sigma$~$\alpha$-deduces that~$\Phi(n)=i$.
\item We say that \emph{$\sigma$~deduces (or~$\alpha$-deduces) the value of~$\Phi(n)$} if there is some~$i$ such that~$\sigma$~deduces (or~$\alpha$-deduces) that~$\Phi(n)=i$.
\end{enumerate}
\end{definition}

\begin{lemma}\label{L:deducebounded}
If~$\sigma$ deduces that~$\Phi(n)=i$ then there exists some~$\alpha<\omega_1^{CK}$ such that~$\sigma$~$\alpha$-deduces that~$\Phi(n)=i$.

\end{lemma}

\begin{proof}
We first show that the set of finite partial oracles~$\sigma$ such that~$\sigma$ deduces that~$\Phi(n)=i$ can be defined by an arithmetically-definable monotonic closure operator, and then appeal to the fact that all such closure operators reach their limit at a stage before~$\omega_1^{CK}$.

Define an operator~$\Gamma$ that takes a set,~$X$, of finite partial oracles to the set of finite partial oracles~$\sigma$ such that one of the following holds:
\begin{itemize}
\item~$\sigma \in X$;
\item~$X$ contains infinitely many~$1$-extensions of~$\sigma$;
\item~$\Phi^\sigma(n) \downarrow = i$.
\end{itemize}
Then,~$\Gamma(X)$ is arithmetic (in fact,~$\Pi^0_2$) in~$X$, and obviously monotonic.

For all ordinals~$\alpha$, let~$\Gamma_\alpha$ denote the~$\alpha$th iteration of~$\Gamma$ on the empty set. (Formally,~$\Gamma_0 = \Gamma(\emptyset)$;~$\Gamma_{\alpha + 1} = \Gamma(\Gamma_\alpha)$; and for~$\alpha$ a limit,~$\Gamma_\alpha = \bigcup_{\beta < \alpha} \Gamma_\beta$.) Then it is easily checked that~$\Gamma_\alpha$ is precisely the set of partial oracles that~$\alpha$-deduce that~$\Phi(n) = i$. Now by Spector's boundedness theorem (cf. \cite[Corollary~5.6]{Sacks-1990}), there is an~$\alpha<\omega_1^{CK}$ such that~$\Gamma_\alpha=\Gamma_{\alpha+1}$. Then a partial oracle~$\sigma$ deduces that~$\Phi(n)=i$ only if~$\sigma$~$\alpha$-deduces that~$\Phi(n)=i$, as desired.
\end{proof}

\begin{lemma}\label{L:deducebounds}
If~$\alpha<\omega_1^{CK}$, then~$\emptyset^{(2\alpha+2)}$ computes the set of~$\langle n,i\rangle$ such that~$\sigma$~$\alpha$-deduces that~$\Phi(n)=i$.
\end{lemma}

\begin{proof}

We prove, by induction on~$\alpha$, the following statement. If~$\alpha<\omega_1^{CK}$ and~$a$ is any ordinal notation for~$2\alpha+2$, then whether or not~$\sigma$~$\alpha$-deduces that~$\Phi(n)=i$ is uniformly computable in~$\sigma$,~$\Phi$,~$n$,~$i$,~$a$, and~$H_a$. The uniformity and the ordinal notations will be needed for the limit stage of the proof.

If~$\alpha=0$, then to determine whether or not~$\sigma$~$\alpha$-deduces that~$\Phi(n)=i$, we simply ask whether~$\Phi^\sigma(n)$ halts and is equal to~$i$. This is computable in~$\emptyset'$, and so also in~$\emptyset''$.

Next, suppose~$\alpha = \beta + 1$, and assume the result holds for~$\beta$. To determine whether or not~$\sigma$~$\alpha$-deduces that~$\Phi(n)=i$, we ask whether for infinitely many~$k$ not in the domain of~$\sigma$ there is a~$\tau \succ_1 \sigma$ with~$k$ in its domain that~$\beta$-deduces that~$\Phi(n)=i$. By induction, whether or not~$\tau$~$\beta$-deduces that~$\Phi(n)=i$ is uniformly computable in~$H_b$, where~$b$ is a notation for~$2\beta+2$. Therefore, the question of whether~$\sigma$~$\alpha$-deduces that~$\Phi(n)=i$ is uniformly computable in~$(H_b)''$, and hence in~$H_a$ for~$a$ any notation for~$|b|+2 = 2\beta + 4 = 2\alpha + 2$.

If~$\alpha$ is a limit, let~$u=3\cdot5^e$ with~$|u|=\alpha$. To determine whether or not~$\sigma$~$\alpha$-deduces that~$\Phi(n)=i$, we ask whether for infinitely many~$k$ not in the domain of~$\sigma$ there is a~$v$ enumerated by~$\Phi_e$ and a~$\tau\succ_1\sigma$ with~$k$ in its domain that~$|v|$-deduces that~$\Phi(n)=i$. For each~$v$ enumerated by~$\Phi_e$, we have~$|v|<\alpha$, and so~$2|v|+2<\alpha$ since~$\alpha$ is a limit. Therefore,~$H_u$ can uniformly compute whether or not~$\tau$~$|v|$-deduces that~$\Phi(n)=i$., so~$(H_u)''$ can compute whether or not the number of such~$\tau$ goes to infinity over all~$v$ enumerated by~$\Phi_e$.

There are infinitely many~$\tau\succ_1\sigma$ that~$\beta$-deduce that~$\Phi(n)=i$ for some~$\beta<\alpha$, if and only if there are infinitely many~$\tau\succ_1\sigma$ that~$|v|$-deduce that~$\Phi(n)=i$ for some~$v$ enumerated by~$\Phi_e$. Thus~$(H_u)''$ can uniformly compute whether or not~$\sigma$~$\alpha$-deduces that~$\Phi(n)=i$.
\end{proof}

\begin{lemma}\label{L:deducecorrect}
Suppose~$B \leq_\ii A$ via~$\Phi$, and that~$A \not\from_\ii \emptyset^{(\alpha)}$ for every computable ordinal~$\alpha$. If a partial oracle~$\sigma$ for~$A$ deduces that~$\Phi(n)=i$, then~$B(n)=i$.
\end{lemma}

\begin{proof}
We prove by induction on~$\alpha$ that if~$\sigma$ is a finite partial oracle for~$A$ that~$\alpha$-deduces that~$\Phi(n)=i$ then~$B(n)=i$. If~$\alpha = 0$, then~$\sigma$~$\alpha$-deducing that~$\Phi(n) = i$ means that~$\Phi^\sigma(n) \downarrow = i$, so~$B(n) = i$ since~$\sigma$ is a partial oracle for~$A$ and~$\Phi$ witnesses an infinite-information reduction from~$A$ to~$B$.

Now suppose~$\alpha>0$, and assume the result for all~$\beta < \alpha$. If~$\sigma$ is a partial oracle for~$A$ that~$\alpha$-deduces that~$\Phi(n) = i$, then there are infinitely many~$\tau \succ_1 \sigma$ that~$\beta$-deduce that~$\Phi(n) = i$ for some~$\beta < \alpha$. By assumption, if any of these~$\tau$ is a partial oracle for~$A$, it follows that~$B(n) = i$. So suppose not. Then since~$\sigma$ was a finite partial oracle for~$A$, we know that any~$\tau$ that~$\beta$-deduces that~$\Phi(n) = i$ for some~$\beta < \alpha$ must disagree with~$A$ on the extra bit it has in its domain that~$\sigma$ does not. But then by Lemma \ref{L:deducebounds}, we can uniformly compute infinitely many such~$\tau$ using~$\emptyset^{(2\alpha+2)}$. Hence~$\emptyset^{(2\alpha+2)}$ can compute infinitely many incorrect bits of~$A$ and thus infinitely many correct bits of~$A$, contradicting the fact that~$A \not\from_\ii \emptyset^{(2\alpha+2)}$.
\end{proof}

The proof of Lemma \ref{L:maxpair2} now follows.

\begin{proof}[Proof of Lemma \ref{L:maxpair2}]
Assume that~$A\geq_\ii B$ via~$\Phi$, and assume that for every~$\alpha<\omega_1^{CK}$, we have that~$B \not\from_\ii \emptyset^{(\alpha)}$, and~$A \not\from_\ii B\oplus\emptyset^{(\alpha)}$. For each finite partial oracle~$\sigma$, let~$f_\sigma$ be the partial function where~$f_\sigma(n)=i$ if and only if~$\sigma$ deduces that~$\Phi(n)=i$. It is straightforward to check (by induction) that~$f_\sigma$ is well-defined.

We claim there exists a finite partial oracle~$\sigma$ for~$A$ such that for every~$\tau \succ_1 \sigma$ that is also a partial oracle for~$A$, the domain of~$f_\tau$ is strictly larger than the domain of~$f_\sigma$. Suppose not, and define a sequence~$\sigma_0 \prec_1 \sigma_1 \prec_1 \cdots$ of finite partial oracles for~$A$, as follows. Let~$\sigma_0$ be the empty oracle, and given~$s$, let~$\sigma_{s+1}$ be any partial oracle for~$A$ such that~$\sigma_{s+1} \succ_1 \sigma_s$ and~$f_{\sigma_{s+1}} = f_{\sigma_s}$. Then~$\bigcup_s \sigma_s$ is a partial oracle~$(A)$ for~$A$ with infinite domain such that if~$\Phi^{(A)}(n) \downarrow = i$ then~$\sigma_0$ deduces that~$\Phi(n) = i$. But~$\sigma_0$ cannot deduce the values of~$\Phi(n)$ for infinitely many~$n$, because otherwise there would be some~$\alpha<\omega_1^{CK}$ such that~$\emptyset^{(\alpha)}$ could compute infinitely many bits of~$B$. 

This is because, for every~$\alpha\omega_1^{CK}$, we have that~$A \not\from_\ii \emptyset^{(\alpha)}$, so by Lemma \ref{L:deducecorrect} all the facts that~$\sigma_0$ can deduce are correct for~$B$. Also, by Lemma \ref{L:deducebounded}, each instance of deduction happens at an ordinal~$<\omega_1^{CK}$, and so, if~$\sigma_0$ could deduce the values of~$\Phi(n)$ for infinitely many~$n$, there would be some ordinal,~$\alpha<\omega_1^{CK}$ such that~$\sigma_0$ could~$\alpha$-deduce the values of~$\Phi(n)$ for infinitely many~$n$. By Lemma \ref{L:deducebounds}, we would then have that~$\emptyset^{(2\alpha+2)}$ could correctly compute infinitely many bits of~$B$, contracticting that for every~$\alpha<\omega_1^{CK}$, we have that~$B \not\from_\ii \emptyset^{(\alpha)}$.

It follows that the domain of~$\Phi^{(A)}$ must be finite, because~$\Phi^{(A)}$ cannot halt on any~$n$ such that~$\sigma_0$ cannot deduce the value of~$\Phi(n)$. This contradicts that~$\Phi$ witnesses an infinite-information reduction from~$A$ to~$B$. Thus the claim is proved.

So fix a~$\sigma$ as given by the claim. Consider the set of all~$\tau\succ_1\sigma$ such that~$f_\tau$ does not incorrectly compute any bits of~$B$. (Thus, this set includes every~$1$-extension of~$\sigma$ that is a partial oracle for~$A$, but it may include other~$1$-extensions.) Let~$X$ be the set of numbers that are in the domain of~$f_\tau$ for some such~$\tau$ but not in the domain of~$f_\sigma$. We consider three cases, each of which leads to a contradiction.

\medskip
\noindent \emph{Case~1: There are infinitely many~$\tau \succ_1 \sigma$ such that~$f_\tau$ disagrees with~$B$ on some bit.} In this case, there must be some~$\alpha<\omega_1^{CK}$ such that there are infinitely many~$\tau \succ_1 \sigma$ that incorrectly deduce the value of~$\Phi(n)$ for some~$n$. Thus,~$A \from_\ii B\oplus\emptyset^{(2\alpha+2)}$.

 This is because~$B\oplus\emptyset^{(2\alpha+2)}$ can search for~$\tau \succ_1 \sigma$ for which~$\tau$~$\alpha$-deduces a false output, and whenever it finds such a~$\tau$, it knows that the extra bit in the domain of~$\tau$ disagrees with~$A$. Thus,~$B\oplus\emptyset^{(2\alpha+2)}$ can compute infinitely many bits of~$A$, and we have in particular a contradiction to the assumption that for each~$\alpha\omega_1^{CK}$,~$A \not\from_\ii B \oplus \emptyset^{(\alpha)}$.

\medskip
\noindent \emph{Case~2: There are only finitely many~$\tau \succ_1\sigma$ such that~$f_\tau$ disagrees with~$B$ on some bit, but~$X$ is infinite.} In this case, there must be some~$\alpha<\omega_1^{CK}$ such that there are infinitely many~$n$ such that~$\tau$~$\alpha$-deduces the value of~$\Phi(n)$. Thus,~$B \leq_\ii \emptyset^{(2\alpha+2)}$.

This is because for almost all~$\tau \succ_1 \sigma$, any output by~$f_\tau$ agrees with~$B$, and by assumption, there are infinitely many~$n$ such that~$\tau$~$\alpha$-deduces the value of~$\Phi(n)$ for some such~$\tau$. If~$\tau$~$\alpha$-deduces that~$\Phi(n)=i$, then~$f_\tau(n)=i$, so in particular, since only finitely many of the deduced facts can be false, we~$\emptyset^{(2\alpha+2)}$ can search for~$\tau\succ_1\sigma$ that~$\alpha$-deduce any facts other than those finitely many facts, Thus, we have a contradiction to the assumption that for each~$\alpha\omega_1^{CK}$,~$B \not\from_\ii \emptyset^{(2\alpha+2)}$.

\medskip
\noindent \emph{Case~3: Otherwise.} In this case, by the pigeonhole principle, there must be an~$n\in X$ such that~$f_\tau(n) \downarrow$ for infinitely many~$\tau\succ_1\sigma$. Furthermore, it must be that~$f_\tau(n) = B(n)$ for all such~$\tau$, since otherwise, we would be in Case~1. In other words, letting~$i = B(n)$, we have that infinitely many~$1$-extensions of~$\sigma$ deduce that~$\Phi(n) = i$. But by definition, this means that~$\sigma$ also~deduces that~$\Phi(n) = i$, so~$n$ could not have been in~$X$, a contradiction.

\medskip
The proof is complete.
\end{proof}

\noindent
We now present a pair of questions concerning joins and maximal pairs in the infinite information degrees.

The first question concerns whether the conditions Lemma \ref{L:maxpair1} are required to prove Theorem \ref{T:maxpair}. 

\begin{question}
If $B_0$, $B_1$ are mutually 1-random, or mutually 1-generic, do they necessarily form a maximal pair for $\leq_\ii?$
\end{question}

\noindent
By the proof of Theorem \ref{T:maxpair}, mutually $\Delta_1^1$-random, or mutually $\Delta_1^1$-generic would be sufficient, but we do not know whether the hyperarithmetic machinery of $\alpha$-deduction is required for the result.

Our second question, posed by Noah Schweber (private communication), concerns whether joins ever exist in the infinite information degrees.

\begin{question}\label{Q:iijoins}

Do there exist reals $A,B,C$ such that the following hold?

\begin{itemize}
\item $A\nleq_\ii B$, and $B\nleq_\ii A$.
\item $A\leq_\ii C$, and $B\leq_\ii C$.
\item For any $D$, if $A\leq_\ii D$, and $B\leq_\ii D$, then $C\leq_\ii D$.

\end{itemize}

\end{question}

\noindent
We have shown that under certain hypotheses, joins can fail to exist in the most spectacular way possible. However, we do not know whether under other conditions, joins can exist.

It is tempting to attempt to answer Question \ref{Q:iijoins} by letting $A=U\oplus W$, $B=V\oplus W$, $C=W$. This would certainly make $A\leq_\ii C$, and $B\leq_\ii C$. Some finesse in choosing $U,V,W$ can ensure that $A\nleq_\ii B$, and $B\nleq_\ii A$. However, the third point seems quite difficult to address.

\printbibliography

\end{document}